\documentclass{amsart}
\usepackage{amsmath, amsthm, amsfonts, amssymb}
\usepackage{mathrsfs}
\usepackage{microtype}
\usepackage[utf8]{inputenc}
\providecommand{\noopsort[1]{}}
\usepackage{bbm}
\usepackage{dsfont}
\usepackage{ifthen}
\usepackage{nicefrac}
\numberwithin{equation}{section}
\usepackage{a4}
\usepackage[active]{srcltx}
\usepackage{verbatim}
\usepackage{hyperref}
\usepackage{color,bm}

\usepackage[shortlabels]{enumitem}
\setlist{leftmargin=*}
\setlist[1]{labelindent=1.2\parindent}

\allowdisplaybreaks

\newtheorem{thm}{Theorem}[section]
\newtheorem{cor}[thm]{Corollary}
\newtheorem{prop}[thm]{Proposition}
\newtheorem{lem}[thm]{Lemma}

\theoremstyle{remark}
\newtheorem{rem}[thm]{Remark}

\newtheorem{hyp}[thm]{Hypotheses}
\newtheorem{example}[thm]{Example}

\theoremstyle{definition}

\newcommand{\coloneqq}{\mathrel{\mathop:}=}

\renewcommand{\Re}{{\rm Re}\,}

\newcommand{\eps}{\varepsilon}

\newcommand{\CR}{\mathds{R}}
\newcommand{\CC}{\mathds{C}}
\newcommand{\K}{\mathds{K}}
\newcommand{\CN}{\mathds{N}}

\newcommand{\cL}{\mathscr{L}}

\newcommand{\tnorm}[1]{{|\kern-0.3ex|\kern-0.3ex|#1|\kern-0.3ex|\kern-0.3ex|}}

\newcommand{\la}{\langle}
\newcommand{\ra}{\rangle}
\DeclareMathOperator{\tr}{\mathrm{tr}}

\let\div\undefined
\DeclareMathOperator{\div}{\mathrm{div}}

\begin{document}
\title{$\bm{L^p}$-theory for Schr\"odinger systems}
\author{M. Kunze}
\address{Universit\"at Konstanz, Fachbereich Mathematik und Statistik, 78457 Konstanz, Germany}
\email{markus.kunze@uni-konstanz.de}
\author{L. Lorenzi}
\address{Dipartimento di Scienze Matematiche, Fisiche e Informatiche, Edificio di Matematica e Informatica, Universit\`a di Parma, Parco Area delle Scienze 53/A, I-43124 Parma, Italy.}
\email{luca.lorenzi@unipr.it}
\author{A. Maichine}
\address{Dipartimento di Matematica, Universit\`a degli Studi di Salerno, Via Giovanni Paolo II 132, I-84084 Fisciano (SA), Italy}
\email{amaichine@unisa.it}
\author{A. Rhandi}
\address{Dipartimento di Ingegneria dell'Informazione, Ingegneria Elettrica e Matematica Applicata, Università degli Studi di Salerno, Via Ponte Don Melillo 1, 84084 Fisciano (Sa), Italy}
\email{arhandi@unisa.it}
\date{}
\begin{abstract}
In this article we study for $p\in (1,\infty)$ the $L^p$-realization of the vector-valued Schr\"odinger
operator $\mathscr{L}u \coloneqq \div (Q\nabla u) + V u$. Using a noncommutative version of the Dore--Venni theorem due to
Monniaux and Pr\"uss, we prove that the $L^p$-realization of $\mathscr{L}$, defined on the intersection of the natural domains of the differential
and multiplication operators which form ${\mathscr L}$, generates
a strongly continuous contraction semigroup on $L^p(\CR^d; \CC^m)$. We also study additional properties of
the semigroup such as extension to $L^1$, positivity, ultracontractivity and prove that the generator has compact resolvent.
\end{abstract}
\maketitle

\section{Introduction}
{Second-order} elliptic differential operators with unbounded coefficients appear naturally as infinitesimal generators of  diffusion processes; the associated parabolic equation is then the Kolmogorov equation for that process. While the \emph{scalar} theory
of such equations is by now well developed (see \cite{lb07} and the references therein), the literature on \emph{systems} of parabolic equations with unbounded coefficients is still sparse.
{Beside their own interests, such systems appear naturally in the study of backward-forward stochastic differential systems, in the study of Nash equilibria to stochastic differential games, {in the analysis of the weighted $\overline{\partial}$-problem in $\CC^d$, in the time-dependent Born--Openheimer theory
and also in the study of Navier--Stokes equations. We refer the reader to \cite[Section 6]{aalt}, \cite{Ha-He}, \cite{Dall}, \cite{BGT}, \cite{hieber}, \cite{HRS} and \cite{Ha-Rh}}  for further details.}

One of the first articles concerned with systems of parabolic equations with unbounded coefficients is \cite{hetal09} where the diffusion coefficients were assumed to be strictly elliptic and bounded and coupling
between the equations was through a potential term $V$ and, additionally, through an unbounded drift term $F$. It should be noted
that for $V=0$ and a drift term growing as $|F(x)| \asymp |x|^{1+\eps}$ one can not expect generation of a semigroup on $L^p$ with respect to Lebesgue measure, even in the scalar case, see \cite{prs06}. Consequently, the drift term in \cite{hetal09} may not grow like $|x|^{1+\eps}$, whereas a growth like $|x|\log(1+|x|)$ is possible. Due to the interaction between drift and potential term, there are additional assumptions on the potential which in absence of a drift term are somewhat restrictive. Indeed, for symmetric potentials, the assumptions made in \cite{hetal09} imply the boundedness of the potential term; as for antisymmetric potential terms, the entries may
grow logarithmically.

Subsequently, there were some other publications \cite{aalt, alp16, dl11} with less restrictive assumptions on the coefficients; in particular, also unbounded diffusion coefficients can be considered. The strategy in these references is quite different from
that in \cite{hetal09}. Namely, in \cite{aalt, alp16, dl11} solutions to the parabolic equation are at first constructed in the space of
bounded and continuous functions. Afterwards the semigroup is extrapolated to the $L^p$-scale. Consequently, even though this approach allows for more general coefficients, we obtain no information about the domain of the generator of the semigroup -- a crucial information {for applications}.

In this article, we follow the strategy from \cite{hetal09} in using a noncommutative Dore--Venni theorem due to
Monniaux and Pr\"uss \cite{mp97}, thereby obtaining the domain of the generator explicitly. As we have no drift term, we can allow
much more general potential terms; in particular, we can have potential terms whose entries grow like
$|x|^r$ for some $r\in [1,2)$.

This article is organized as follows. In Section \ref{s.hrp} we fix our assumptions, present some examples satisfying this assumptions (and some that do not) and recall some preliminary results that will be used subsequently. Section \ref{s.gen} contains the actual generation theorem and in the concluding Section \ref{s.further} we study further properties of the semigroup.

\subsection*{Notation}
For any natural number $k$ we denote the Euclidean norm on $\CR^k$ or $\CC^k$ by $|\cdot|$ and the Euclidean inner-product
by $\la \cdot, \cdot \ra$. Now, let $d, m \in \CN$, $\Omega\subset \CR^d$ be an open set and $\K = \CR$ or
$\K=\CC$. By $C_c^\infty(\Omega; \K^m)$ we denote the space of test functions, i.e.\
the space of all infinitely differentiable functions with compact support.
For $p \in [1,\infty]$, $L^p(\Omega; \K^m)$ refers to the classical, vector-valued Lebesgue space of $p$-integrable functions (essentially bounded functions if $p=\infty$) on $\Omega$. The norm on $L^p(\CR^d; \CR^m)$ is denoted
by $\|\cdot\|_p$; for $p\in [1,\infty)$, the conjugate index is denoted by $p'$, i.e.\
$p^{-1} + p'^{-1}=1$ and $\langle \cdot, \cdot \rangle_{p,p'}$ denotes the canonical dual pairing between
$L^p(\CR^d; \K^m)$ and $L^{p'}(\CR^d; \K^m)$. Given $k \in \CN$,
$W^{k,p}(\Omega; \K^m)$ is the classical Sobolev space of order $k$, i.e. the space of all functions
$f \in L^p(\Omega; \K^m)$ such that for every multiindex $\alpha =(\alpha_1, \ldots, \alpha_d) \in \CN_0^d$
with $\alpha_1+\cdots + \alpha_d \leq k$ the distributional derivative {$\partial^\alpha f:=\frac{\partial^{\alpha} f}{\partial x^{\alpha}}$} belongs to $L^p(\Omega;\K^m)$.
The norm in $W^{k,p}(\CR^d;\K^m)$ is denoted by $\|\cdot\|_{k,p}$.
The space $W^{k,p}_\mathrm{loc}(\Omega; \K^m)$ consists of those measurable and locally integrable functions which, along
with their distributional derivatives up to order $k$, belong locally to $L^p$. In the case where $m=1$, we drop $\K^m$
from our notation, i.e.\ we write $L^p(\CR^d)$, $W^{k,p}(\CR^d)$, etc.

\section{Hypotheses, remarks and preliminaries}\label{s.hrp}

Throughout, we make the following assumptions.
\begin{hyp}\label{hyp1}
Let $d,m \in \CN$.
\begin{enumerate}
[(a)]
\item Let $Q = (q_{i,j}) : \CR^d\to \CR^{d\times d}$ be a symmetric matrix-valued function with Lipschitz continuous
entries such that
\begin{eqnarray*}
\eta_1|\xi|^2 \leq \langle Q(x)\xi,\xi\rangle \leq \eta_2|\xi|^2
\end{eqnarray*}
for all $x,\xi\in \CR^d$ and some constants $\eta_1, \eta_2>0$. For $p\in (1,\infty)$ we define the operator $A_p$ on $L^p(\CR^d; \CC^m)$ by
$D(A_p) = W^{2,p}(\CR^d; \CC^m)$ and
\begin{eqnarray*}
A_p u = \big[\div (Q \nabla u_k) - u_k\big]_{k=1,\ldots, m}.
\end{eqnarray*}
\item Let $V: \CR^d \to \CR^{m\times m}$ be a matrix-valued function with entries in $W^{1,\infty}_{\mathrm{loc}}(\CR^d)$
such that
\begin{equation}\label{eq.diss}
\langle V(x)\xi, \xi\rangle \leq -|\xi|^2
\end{equation}
for all $x\in\CR^d$ and $\xi \in \CR^m$.
Moreover, assume that there exists a constant $\alpha\in [0,\frac{1}{2})$ such that for every $j=1, \ldots, d$
the matrix-valued function $x\mapsto D_jV(x)(-V(x))^{-\alpha}$ is uniformly bounded in $\CR^d$.
For $p\in (1,\infty)$ we define the operator $V_p$ on $L^p(\CR^d; \CC^m)$
by setting $D(V_p) = \{f \in L^p(\CR^d; \CC^m) : Vf \in L^p(\CR^d;\CC^m)\}$ and $V_pf \coloneqq Vf$. Here $Vf$ is to be understood
as matrix-vector product.
\end{enumerate}
\end{hyp}

\begin{rem}\label{rem1}
Our assumptions imply that both $A_p$ and $V_p$ are injective operators. This was done for ease of notation. More generally,
we could also allow potentials $\tilde V$ which satisfy
\begin{equation}
\langle \tilde V(x)\xi, \xi \rangle \leq \beta |\xi|^2,
\label{torrione}
\end{equation}
for some $\beta \geq 0$ and all $\xi \in \CR^m$. Indeed, shifting the potential $\tilde V(x)$ by $(\beta +1)I$, we obtain
a potential $V(x) = \tilde V(x) -(\beta+1)I$ which then satisfies Estimate \eqref{eq.diss}.
In a similar way we can also compensate the entries of $u$ which were subtracted on the diagonal in
the definition of $A_p$. Note that shifting the potential corresponds to a rescaling of the semigroup.

In this more general situation, we cannot expect $\tilde V(x)$ to be invertible, whence the assumption
that $D_j\tilde V(x) (-\tilde V(x))^{-\alpha}$ to be uniformly bounded does not make sense. It has to be assumed
for the \emph{shifted} potential $\tilde V(x) - (\beta + 1)I$. Note, however, that shifting the potential does not change its derivative. In some concrete situations where the potential $\tilde V(x)$ is invertible, see Example \ref{ex.poly} below, the boundedness condition for the shifted potential is equivalent to that of the unshifted one.
\end{rem}

We now present an example which shows that without a semiboundedness assumption on $V$ as in \eqref{torrione} we cannot
expect generation of a semigroup in general.

\begin{example}
We consider the situation where $d=1$ and $m=2$.
 Let $\mathscr{L}$ be the vector-valued operator defined on smooth functions $\zeta : \CR \to \CR^2$ by
$\mathscr{L}\zeta=\zeta''+V\zeta$, where
\begin{eqnarray*}
V(x)=
\begin{pmatrix}
0 & x\\
0 & 0
\end{pmatrix},\qquad\;\,x\in\CR.
\end{eqnarray*}

Obviously,  the quadratic form $\xi\mapsto\la V(x)\xi, \xi\ra$ takes for $x\neq 0$ arbitrary values
in $\CR$ so that $V$ does not satisfy the semiboundedness assumption \eqref{torrione}. Fix $p\in (1,\infty)$.
We are going to prove that no realization in $L^p(\CR; \CR^2)$ of the operator $\mathscr{L}$ generates a semigroup. To that end, it suffices to prove that, for every $\lambda>0$ and properly chosen $f \in L^p(\CR;\CR^2)$, the resolvent equation
$\lambda u-\mathscr{L}u=f$ does not admit solutions in the maximal
domain $D_{p,\max}(\mathscr{L})=\{u\in L^p(\CR;\CR^2): \mathscr{L} u\in L^p(\CR;\CR^2)\}$.
The resolvent equation can be rewritten as a system as follows:
\begin{eqnarray*}
\left\{
\begin{array}{ll}
\lambda u_1(x)-u_1''(x)-xu_2(x)=f_1(x), &x\in\CR,\\[1mm]
\lambda u_2(x)-u_2''(x)=f_2(x), & x\in\CR.
\end{array}
\right.
\end{eqnarray*}

For simplicity, we will only consider functions $f_1, f_2$ which are supported in $[1, \infty)$. Solving the second equation
in $L^p(\CR)$, we find that the unique solution $u_2$ is given by
\begin{eqnarray*}
u_2(x) = \frac{1}{2\sqrt{\lambda}} \int_x^\infty e^{\sqrt{\lambda}(x-t)} f_2(t)\, dt + \frac{1}{2\sqrt{\lambda}} \int_1^x
e^{-\sqrt{\lambda}(x-t)}f_2(t)\, dt + c e^{-\sqrt{\lambda}x}
\end{eqnarray*}
for $x\ge 1$ and $u_2(x)=0$ for $x< 1$. The constant $c$ is chosen such that $u_2(1) = 0$ so that $u_2$ is a continuous function.

From now on, we pick $f_2(t) = t^{-1}$ for $t \geq 1$ and $f_2(t) = 0$ for $t< 1$. It is then easy to see that $xu_2(x)$ converges to $\lambda^{-1}$ as $x \to \infty$. In particular, $x u_2(x) \geq (2\lambda)^{-1}$ for large enough $x$, say $x \geq x_0$. Inserting this into the first equation and choosing $f_1 \equiv 0$, we obtain the differential inequality $u_1'' \leq \lambda u_1 - (2\lambda)^{-1}$.

Integrating this inequality, we obtain first
\begin{eqnarray*}
u_1'(x) \leq c_{1,\lambda} + \lambda\int_{x_0}^x u_1(t)\, dt - \frac{x}{2\lambda},\qquad\;\,x\ge x_0,
\end{eqnarray*}
and then
\begin{equation}\label{eq.contr}
u_1(x) \leq c_{2,\lambda} + c_{1,\lambda}x + \lambda \int_{x_0}^x\int_{x_0}^t u_1(s)\, ds\, dt - \frac{x^2}{4\lambda},\qquad\;\,x\ge x_0,
\end{equation}
for certain constants $c_{1,\lambda}, c_{2,\lambda}$. Suppose now that our resolvent equation has a solution
$(u_1, u_2) \in L^p(\CR; \CR^2)$. As $u_1 \in L^p(\CR)$, we can use H\"older's inequality to estimate
\begin{eqnarray*}
\bigg|\int_{x_0}^x \int_{x_0}^t u_1(s)\, ds\, dt\bigg| \leq \|u_1\|_p\int_{x_0}^x t^{1-\frac{1}{p}}\, dt
= c_3x^{2-\frac{1}{p}} + c_4
\end{eqnarray*}
for all $x\geq x_0$. Inserting this into \eqref{eq.contr} and letting $x\to \infty$ we obtain that $u_1(x)$ diverges to $-\infty$
for $x \to \infty$, which contradicts the condition $u_1 \in L^p(\CR)$.
\end{example}

We are next going to illustrate that Hypotheses \ref{hyp1} allow for potentials $V$ whose entries grow
more than linearly at infinity.

\begin{example}\label{ex.poly}
We again consider the situation where $d=1$ and $m=2$. Choosing  $r \in [1,2)$, we set
\begin{eqnarray*}
V(x) \coloneqq \begin{pmatrix}
0 & {1+|x|^r}\\
-{(1+|x|^r)} & 0
\end{pmatrix}
= {(1+|x|^r)} \begin{pmatrix}
0 & 1\\
-1 & 0
\end{pmatrix},\qquad\;\,x\in\CR^d.
\end{eqnarray*}
As $V(x)$ is antisymmetric, we find $\la V(x)\xi, \xi\ra =0$ for all $x \in \CR$ and $\xi \in \CR^2$. Note that, by rescaling, we can
arrange that the quadratic form is bounded from above by $-1$, cf.\ Remark \ref{rem1}. Using that antisymmetric matrices are diagonalizable, we see that
\begin{eqnarray*}
(-V(x))^{-\alpha} = (1+|x|^r)^{-\alpha} \begin{pmatrix}
0 & -1\\
1 & 0
\end{pmatrix}^{-\alpha},\qquad\;\,x\in\CR,
\end{eqnarray*}
so that
\begin{eqnarray*}
D_xV(x) \cdot (-V(x))^{-\alpha} = r|x|^{r-2}(1+|x|^r)^{-\alpha}x\begin{pmatrix}
0 &1\\
-1 & 0
\end{pmatrix}\cdot \begin{pmatrix}
0 & -1\\
1 & 0
\end{pmatrix}^{-\alpha}
\end{eqnarray*}
for all $x\in\CR$.
Now, if we pick $\alpha \in (\frac{r-1}{r}, \frac{1}{2})$, then we have $r-1-\alpha r <0$, so that the function $x\mapsto D_x V(x)\cdot (-V(x))^{-\alpha}$ is indeed bounded. Note that in this situation, all matrices $V(x)$ are simultaneously diagonalizable
so that we are basically in a scalar situation. Thus, the established boundedness is stable under shifting.
\end{example}

We next establish some properties of the operators $A_p$ and $V_p$. Let us recall that an operator $A$ on a Banach space $X$
is called \emph{sectorial} if it is closed, densely defined and there exists an angle $\varphi \in (0,\pi]$ such that the sector
$\Sigma_\varphi \coloneqq \{ z\in \CC : z \neq 0, |\arg z| < \varphi\}$ is contained
in the resolvent set of $-A$ and $M_\varphi \coloneqq \sup_{\lambda \in \Sigma_\varphi} \|\lambda (\lambda + A)^{-1}\|_{\cL(X)} < \infty$.
The \emph{spectral angle} $\varphi_A$ of a sectorial operator $A$ is defined as
\begin{eqnarray*}
\varphi_A \coloneqq \inf\{ \varphi \in [0, \pi) : \Sigma_{\pi-\varphi} \subset \rho (-A) \mbox{ and } M_{\pi-\varphi} < \infty\}.
\end{eqnarray*}
It is well known, see e.g.\ \cite[Theorem II.4.6]{en00}, that $-A$ generates a bounded analytic semigroup if and only if
it is sectorial with spectral angle $\varphi_A < \frac{\pi}{2}$. The operator $A$ is called \emph{quasi-sectorial} if $\nu+ A$ is sectorial
for some $\nu \geq 0$.

An injective, sectorial operator $A$ is said to admit \emph{bounded imaginary powers} if the closure of
$A^{is}$, initially defined on $D(A)\cap R(A)$, defines a bounded operator on $X$ for all $s \in \CR$ and
the family $(A^{is})_{s \in \CR}$ is a strongly continuous {group}
on $X$. The \emph{power angle} $\theta_A$ of $A$ is the growth bound of this group, i.e.,
\begin{eqnarray*}
\theta_A \coloneqq \inf\{ \omega \geq 0 : \exists\, M : \|A^{is}\| \leq M e^{\omega |s|}\;\, \forall\, s \in \CR\}.
\end{eqnarray*}
By the Pr\"uss--Sohr theorem (\cite{ps90}) we have $\theta_A \geq \varphi_A$. For more information on this topic we refer e.g., to Chapter 3
of \cite{h06}.

We now collect some properties of the operators $A_p$ and $V_p$.

\begin{prop}\label{p.properties}
Let $1<p<\infty$.
\begin{enumerate}
[(a)]
\item The operator $-A_p$ is invertible, sectorial and admits bounded imaginary powers. Its power angle is 0. Consequently, for every
$\vartheta >0$ there exists a constant $c$ such that for $s \in \CR$ and $\lambda \in \Sigma_{\pi-\vartheta}$ we have
\begin{eqnarray*}
\hskip 1.4truecm \|(\lambda - A_p)^{-1}\|_{\mathscr{L}(L^p(\CR^d;{\CC^m}))} \leq \frac{c}{1+|\lambda|},\qquad
\|(-A_p)^{is}\|_{\mathscr{L}(L^p(\CR^d;{\CC^m}))} \leq ce^{\vartheta |s|}.
\end{eqnarray*}
\item The operator $-V_p$  is {invertible} and admits bounded imaginary powers. Its power angle is at most $\frac{\pi}{2}$. Consequently,
for every
$\vartheta >\frac{\pi}{2}$ there exists a constant $c$ such that for $s \in \CR$ and $\lambda \in \Sigma_{\pi-\vartheta}$ we have
\begin{eqnarray*}
\hskip 1.4truecm\|(\lambda - V_p)^{-1}\|_{\mathscr{L}(L^p(\CR^d;{\CC^m}))} \leq \frac{c}{1+|\lambda|},\qquad
\|(-V_p)^{is}\|_{\mathscr{L}(L^p(\CR^d;{\CC^m}))} \leq ce^{\vartheta |s|}.
\end{eqnarray*}
\end{enumerate}
\end{prop}

\begin{proof}
(a) It was proved in \cite[Theorem 6.1]{ds97} that for every $\varphi \in (0,\frac{\pi}{2})$ the operator $-A_p$ has a bounded $H^\infty$-calculus on $\Sigma_{\pi-\varphi}$. As for every $s \in \CR$ the function $f(z) = z^{is}$ is bounded and holomorphic on that sector,
the boundedness of the imaginary powers follows.

(b) It follows from Hypotheses \ref{hyp1}(b) that for every $x \in \CR^d$ the matrix $-V(x)$ defines an m-accretive operator on $\CR^m$. By \cite[Example 2]{ps90} (cf.\ also \cite[Corollary 7.1.8]{h06}) we have $|(-V(x))^{is}| \leq e^{\frac{\pi}{2}|s|}$.
Fix $f\in C^{\infty}_c(\CR^d;\CC^m)$ and $s\in\CR$. In view of the Komatsu representation formula (see, e.g., \cite[Proposition 3.2.2]{h06}), it follows that $((-V_p)^{is}f)(x)=(-V(x))^{is}f(x)$ for almost every $x\in\CR^d$. By the above estimate, we can infer that $\|(-V_p)^{is}f\|_p\le e^{\frac{\pi}{2}|s|}\|f\|_p$, which implies the claim by a straightforward density argument.
\end{proof}

\section{The generation result}\label{s.gen}

In this section we are going to prove that the sum $-(A_p+V_p)$, defined on  the domain $D(A_p)\cap D(V_p)$
is closed and quasi-sectorial. To that end, we make use of a non-commutative version of the Dore--Venni Theorem, due to Monniaux and Pr\"uss \cite[Corollary 2]{mp97}. The theorem is valid in arbitrary UMD Banach spaces. We recall that
a Banach space $X$ is called \emph{UMD Banach space} if the Hilbert transform
\begin{eqnarray*}
Hf \coloneqq {\frac{1}{\pi}} \,\mathrm{p.v.}\int_{\CR} f(t-s)\frac{1}{s}\, ds
\end{eqnarray*}
extends to a bounded linear operator on $L^p(\CR;X)$ for one, equivalently, all $p\in (1,\infty)$. In particular,
$X=L^p(\CR^d;\CC^m)$ is a UMD Banach space. For more information we refer the reader to \cite{b86}.

Crucial to apply \cite[Corollary 2]{mp97} is a commutator estimate. To formulate it, we use the following notation.
Given a (sufficiently differentiable) matrix{-valued function} $M : \CR^d \to \CR^{m\times m}$, we write $\nabla^k M$ for the matrix whose $k$-column is the gradient of the $k$-th row of $M$. Thus, if $M=(m_{ij})$, then
\begin{eqnarray*}
\nabla^kM = \begin{pmatrix}
D_1m_{k1} & \ldots & D_1 m_{km}\\
\vdots & \ddots & \vdots \\
D_l m_{k1} & \ldots & D_l m_{km}
\end{pmatrix}.
\end{eqnarray*}

\begin{lem}\label{l.commutator}
Fix $p\in (1,\infty)$, let $A_p$ be defined as in Hypotheses \ref{hyp1} and $M = (m_{ij}) : \CR^d \to \CR^{m\times m}$
be a matrix valued function with entries
in $W^{2,\infty}(\CR^d)$; the induced multiplication operator on $L^p(\CR^d;\CC^m)$ is denoted by $M_p$. Then, for every $f \in W^{2,p}(\CR^d, \CC^m)$
the $k$-th entry of $(A_pM_p - M_pA_p)f$ is given by
\begin{eqnarray*}
\div (Q (\nabla^k M)f) + \tr \big[ Q (\nabla^k M) Df\big].
\end{eqnarray*}
\end{lem}

\begin{proof}
Making use of the definition of the operators and the product rule, we find that the $k$-th entry of $(A_pM_p-M_pA_p)f$ is
\begin{align*}
& \sum_{i,j=1}^d D_i \big( q_{ij} D_j (M f)_k\big) - \sum_{l=1}^m m_{kl} \div (Q \nabla f_l)\\
= & \sum_{i,j=1}^d D_i \Big( q_{ij} D_j \sum_{l=1}^m m_{kl} f_l \Big) - \sum_{l=1}^m m_{kl} \div (Q \nabla f_l)\\
= & \sum_{i,j=1}^d D_i \Big( q_{ij} \sum_{l=1}^m \big[ (D_j m_{kl}) f_l + m_{kl}(D_j f_l)\big]\Big) - \sum_{l=1}^m m_{kl} \div (Q \nabla f_l)\\
= & \div (Q (\nabla^k M)f) + \sum_{l=1}^m \sum_{i,j=1}^d (D_i m_{kl}) q_{ij} (D_j f_l)\\
=& \div (Q (\nabla^k M)f) + \tr \big[ Q (\nabla^k M) Df\big]
\end{align*}
for all $f\in W^{2,p}(\CR^d;\CC^m)$.
\end{proof}

\begin{thm}\label{t.main1}
Assume Hypotheses \ref{hyp1} and fix $p\in (1,\infty)$.
Then, the operator $-(A_p+V_p)$, defined on the domain
$D(A_p)\cap D(V_p)$, is closed, densely defined and quasi-sectorial.
\end{thm}

\begin{proof}
According to Proposition \ref{p.properties} we can pick $\theta_A, \theta_V \in (0,\pi)$ with
$\theta_A+\theta_V < \pi$ such that
\begin{eqnarray*}
\|(\lambda - A_p)^{-1}\|_{\mathscr{L}(L^p(\CR^d;{\CC^m}))}\leq \frac{c}{1+|\lambda|}\quad \mbox{and}\quad
\|(-A_p)^{is}\|_{\mathscr{L}(L^p(\CR^d;{\CC^m}))} \leq ce^{\theta_A |s|}
\end{eqnarray*}
for $\lambda \in \Sigma_{\pi-\theta_A}$ and
\begin{eqnarray*}
\|(\lambda - V_p)^{-1}\|_{\mathscr{L}(L^p(\CR^d;{\CC^m}))} \leq \frac{c}{1+|\lambda|}\quad \mbox{and}\quad
\|(-V_p)^{is}\|_{\mathscr{L}(L^p(\CR^d;{\CC^m}))} \leq ce^{\theta_V |s|}
\end{eqnarray*}
for $\lambda \in \Sigma_{\pi-\theta_V}$.

Fixing $f\in L^p(\CR^d, \CC^m)$, $\lambda \in \Sigma_{\pi-\theta_A}$ and $\mu \in \Sigma_{\pi-\theta_V}$,
we set
\begin{eqnarray*}
C(\lambda, \mu )f \coloneqq (-A_p)(\lambda- A_p)^{-1}\big[ (-A_p)^{-1}(\mu - V_p)^{-1} - (\mu-V_p)^{-1}(-A_p)^{-1}\big]f.
\end{eqnarray*}
We will rewrite this so that we can apply Lemma \ref{l.commutator}. To that end, we approximate the potential $V$ with smoother potentials. Let $(\rho_n)_{n\in \CN}$ be a mollifier sequence and let $\zeta \in C^\infty_c(\CR^d)$ be such that
$0\leq \zeta (x) \leq 1$ for all $x \in \CR^d$ and $\zeta (x) = 1$ for $|x|\leq 1$ whereas $\zeta (x) = 0$ for $|x|\geq 2$.
For a locally integrable function $\varphi$, we set
\begin{eqnarray*}
(K_n\varphi)(x) \coloneqq \zeta \Big(\frac{x}{n}\Big) \int_{\CR^d} \rho_n (y) \varphi (x-y)\, dy,\qquad\;\,x\in\CR^d,\;\,n\in\CN.
\end{eqnarray*}
Clearly, $K_n\varphi \in C_c^\infty(\CR^d)$. As is well-known, $K_n\varphi$ converges locally uniformly to $\varphi$ for every
continuous function $\varphi$ as $n\to\infty$; if $\varphi$ belongs to $W^{j,p}(\CR^d;\CR^m)$ for some $j\in \CN$, then we also get convergence in $W^{j,p}(\CR^d;\CR^m)$.

We now set $V^{(n)} \coloneqq (K_n v_{ij})$ for $n\in\CN$.  Note that ${\langle V^{(n)}\xi, \xi\rangle \leq 0}$ on $\CR^d$ for every $\xi\in\CR^m$. Consequently, the induced multiplication
operator $V_p^{(n)}$ on $L^p(\CR^d;\CC^m)$ is {dissipative} whence for $\mu\in \CC$ with $\Re\mu >0$
we have $\mu \in \rho (V_p^{(n)})$ and $\| (\mu - V_p^{(n)})^{-1}\|_{{\mathscr L}(L^p(\CR^d;\CC^m))} \leq (\Re\mu)^{-1}$. In particular, for fixed $\mu$
the {resolvent operators} $(\mu - V_p^{(n)})^{-1}$ are uniformly bounded. We claim that $(\mu- V_p^{(n)})^{-1}$ converges strongly
to $(\mu - V_p)^{-1}$. Indeed, for $g\in C_c(\CR^d; \CC^m)$ we have
\begin{eqnarray*}
(\mu - V_p^{(n)})^{-1}g - (\mu - V_p)^{-1}g = (\mu- V_p^{(n)})^{-1}(V^{(n)} - V)(\mu - V_p)^{-1}g.
\end{eqnarray*}
Since $(\mu- V_p)^{-1}g$ has compact support and $V^{(n)}$ converges to $V$ locally uniformly on $\CR^d$, it follows that
$(V^{(n)}- V)(\mu - V_p)^{-1}g\to 0$ uniformly and thus in $L^p(\CR^d; \CC^m)$. Using the uniform boundedness of
the resolvents, the claim follows.

Thus, setting
\begin{align*}
&C_{n,m}(\lambda, \mu)f\\
 \coloneqq &A_p(\lambda - A_p)^{-1}\big[ (-A_p)^{-1}(\mu - V^{(n)}_p)^{-1} - (\mu - V_p^{(n)})^{-1} (-A_p)^{-1}\big]A_pK_m((-A_p)^{-1}f)
\end{align*}
we see that, letting first $n$ and then $m$ tend to $\infty$, $C_{n,m}(\lambda, \mu)f$ converges to $C(\lambda, \mu)f$ in $L^p(\CR^d; \CC^m)$. Noting
that $(\mu - V_p^{(n)})^{-1}$ is a multiplication operator with $C^{\infty}$-entries which, together with its derivatives, are bounded, it
follows that we can rewrite $C_{n,m}(\lambda, \mu)f$ as
\begin{eqnarray*}
C_{n,m}(\lambda, \mu)f = (\lambda - A_p)^{-1}\big[ A_p(\mu- V_p^{(n)})^{-1} - (\mu - V_p^{(n)})^{-1}A_p\big]K_m((-A_p)^{-1}f).
\end{eqnarray*}
We can now apply Lemma \ref{l.commutator}. Noting that
\begin{eqnarray*}
D_j (\mu - V^{(n)})^{-1} = (\mu - V^{(n)})^{-1}(D_j V^{(n)}) (\mu - V^{(n)})^{-1}
\end{eqnarray*}
we find  that
\begin{align}
& C_{n,m}(\lambda, \mu)f\notag\\
 = & (\lambda - A_p)^{-1}\div \big( Q (\mu - V^{(n)})^{-1}\nabla V^{(n)} (\mu - V^{(n)})^{-1}K_m((- A_p)^{-1}f)\big)\notag\\
 &+ (\lambda - A_p)^{-1} (\mu- V^{(n)})^{-1} \tr [Q(\nabla V^{(n)})(\mu- V^{(n)})^{-1}\nabla K_m((-A_p)^{-1}f)]\notag\\
  = & (-A_p)^{\frac{1}{2}}(\lambda\!-\! A_p)^{-1}(-A_p)^{-\frac{1}{2}}\div\!\big( Q (\mu\!-\! V^{(n)})^{-1}\nabla V^{(n)} (\mu\! -\! V^{(n)})^{-1}K_m((-A_p)^{-1}f)\big)\notag\\
 &+ (\lambda - A_p)^{-1} (\mu- V^{(n)})^{-1} \tr [Q(\nabla V^{(n)})(\mu- V^{(n)})^{-1}\nabla K_m((-A_p)^{-1}f)].\notag\\
 \label{salerno}
\end{align}
\vskip -3truemm
Here,
\begin{eqnarray*}
\div \big( Q (\mu - V^{(n)})^{-1}\nabla V^{(n)} (\mu - V^{(n)})^{-1} K_m((- A_p)^{-1}f)\big)
\end{eqnarray*}
should be interpreted as the vector, whose $k$-th component is
\begin{eqnarray*}
\div \big( Q (\mu - V^{(n)})^{-1}\nabla^k V^{(n)} (\mu - V^{(n)})^{-1} K_m((- A_p)^{-1}f)\big).
\end{eqnarray*}
The interpretation of the trace term in \eqref{salerno} is similar.

To be able to pass to the limit as $n\to \infty$, we have to take care of the summand involving the divergence. To that end,
 pick $q \in (1,\infty)$
such that $p^{-1}+q^{-1} =1$. Recall
that by the results of \cite{aetal02} the operator $(-A_q)^{-\frac{1}{2}}$ is bounded from
$L^q(\CR^d; \CC^m)$ to $W^{1,q}(\CR^d; \CC^m)$. Therefore, for $j=1, \ldots, d$,
the operator $\big(D_j (-A_q)^{-\frac{1}{2}}\big)^*$
defines a bounded operator on $L^{p}(\CR^d; \CC^m)$. Consequently,
we can extend $(-A_p)^{-\frac{1}{2}}\div$ to a bounded operator $S$ on $L^p(\CR^d;\CC^m)$.
Since the function $K_m((-A_p)^{-1}f$ is compactly supported on $\CR^d$ and $D_j V^{(n)}$ converges to $D_j V$ in $L^p_{\mathrm{loc}}(\CR^d; \CC^m)$ and $D_jV^{(n)}$ is locally uniformly bounded for $j=1, \ldots, d$, by dominated convergence
we deduce that $Q(\mu-V^{(n)})^{-1}\nabla V^{(n)}(\mu-V^{(n)})^{-1}K_m((-A_p)^{-1}f)$ converges to $Q(\mu-V)^{-1}\nabla V(\mu-V)^{-1}K_m((-A_p)^{-1}f)$
in $L^p(\CR^d;\CC^m)$ as $n\to\infty$.

Similarly, we see that
$(\mu- V^{(n)})^{-1} \tr [Q\nabla V^{(n)}(\mu- V^{(n)})^{-1}\nabla K_m((-A_p)^{-1}f)]$ converges to
$(\mu- V)^{-1} \tr [Q\nabla V(\mu- V)^{-1}\nabla K_m((-A_p)^{-1}f)]$. Hence,
letting $n\to \infty$ and denoting by $C_m(\lambda,\mu)f$ the limit, we thus find that
\begin{align*}
C_m(\lambda, \mu)f= & (-A_p)^{\frac{1}{2}}(\lambda - A_p)^{-1}S \big( Q (\mu - V)^{-1}\nabla V (\mu - V)^{-1}(K_m(- A_p)^{-1}f)\big)\\
 &+ (\lambda - A_p)^{-1}  (\mu- V)^{-1} \tr [Q\nabla V(\mu- V)^{-1}\nabla K_m((-A_p)^{-1}f)]\\
 =\colon & T_1 + T_2.
\end{align*}
We are now in the position to provide the crucial commutator estimate. Let us start with the term $T_1$. By the boundedness of the
$H^\infty$-calculus of $-A_p$ the operator $(-A_p)^\frac{1}{2}(\lambda -A_p)^{-1}$ defines a bounded operator on
$L^p(\CR^d; \CC^m)$ and
\begin{eqnarray*}
\|(-A_p)^\frac{1}{2}(\lambda -A_p)^{-1}\|_{\mathscr{L}(L^p(\CR^d; \CC^m))} \leq \frac{C}{|\lambda|^\frac{1}{2}}
\end{eqnarray*}
for a suitable constant $C$. As noted above, $S$ defines a bounded linear operator, as does multiplication
with the bounded matrix-valued function $Q$. By dissipativity, $\|(\mu - V)^{-1}\|_{{\mathscr L}(L^p(\CR^d;\CC^m))}\leq C|\mu|^{-1}$. To estimate
the rest of the term $T_1$, we write
\begin{eqnarray*}
\nabla V(\mu- V)^{-1}K_m((-A_p)^{-1}f) = (\nabla V) \cdot (-V)^{-\alpha} (-V)^{\alpha}(\mu-V)^{-1}K_m((-A_p)^{-1}f),
\end{eqnarray*}
where $\alpha$ is as in Hypotheses \ref{hyp1}. Thus, the term $(\nabla V)\cdot (-V)^{-\alpha}$ is bounded. Using pointwise
the boundedness of the $H^\infty$-calculus of $-V(x)$ (see \cite[Corollary 7.1.8]{h06}), we find
\begin{eqnarray*}
\|(-V)^{\alpha} (\mu - V)^{-1}\|_{\cL (L^p(\CR^d;\CC^M))} \leq \frac{C}{|\mu|^{1-\alpha}}
\end{eqnarray*}
so that, by the boundedness of $(-A_p)^{-1}$ we find that, overall,
\begin{eqnarray*}
\|T_1\|_p \leq \frac{M_1}{|\lambda|^{\frac{1}{2}}|\mu|^{2-\alpha}} \|f\|_p.
\end{eqnarray*}
The term $T_2$ can be estimated similarly, so that, altogether, we have an estimate
\begin{eqnarray*}
\|C(\lambda, \mu) f\|_p \leq \frac{M}{|\lambda|^{\frac{1}{2}}|\mu|^{2-\alpha}} \|f\|_p.
\end{eqnarray*}
We may thus invoke \cite[Corollary 2]{mp97} which yields the claim.
\end{proof}

\begin{cor}\label{c.generate}
The operator $L_p = A_p+V_p$ generates a strongly continuous and contractive semigroup $\{S_p(t)\}_{t\geq 0}$.
\end{cor}

\begin{proof}
It is well-known that the operator $A_p$ is dissipative. Using Estimate \eqref{eq.diss}, we see that also
$L_p$ is dissipative. In view of Theorem \ref{t.main1}, the operator $L_p$ is $m$-dissipative. By the Lumer--Phillips theorem,
see \cite[II Theorem 3.15]{en00},
$L_p$ generates a contraction semigroup.
\end{proof}

\begin{rem}
If $-L_p$ is quasi-sectorial with spectral angle less than $\frac{\pi}{2}$, then the semigroup {$\{S_p(t)\}_{t\ge 0}$ (in the sequel simply denoted by $\{S(t)\}$ to ease the notation)} is analytic, as is well-known, see \cite[II Theorem 4.6]{en00}.
It is a consequence of \cite[Corollary 2]{mp97}, that in the noncommutative version of the Dore--Venni theorem the spectral angle of the sum is at most the maximum of the power angle of the summands. Thus, if the power angle of $-V_p$ is strictly less than
$\frac{\pi}{2}$, which is for example the case when $V(x)$ is symmetric with $\langle V(x)\xi, \xi\rangle \leq -|\xi|^2$
for all $x \in \CR^d$, then $\{S_p(t)\}$ is analytic.
\end{rem}

The following example shows that the semigroup $\{S_p(t)\}$ is not analytic in general.
\begin{example}
Consider the operator $\mathscr{L}$ defined on smooth functions $\zeta : \CR \to \CC^2$ by
$\mathscr{L}\zeta=\zeta''-V\zeta$, where
\begin{eqnarray*}
V(x)=
\begin{pmatrix}
0 & -x\\
x & 0
\end{pmatrix},\qquad\;\,x\in\CR.
\end{eqnarray*}
By Corollary \ref{c.generate}, the realization $L_p$ of $\mathscr{L}$, with domain $W^{2,p}(\CR;\CC^2)\cap D(V_p)$ generates a strongly continuous semigroup on $L^p(\CR;\CC^2)$ for $p\in (1,\infty)$.\\
We diagonalize the matrix $
\begin{pmatrix}
0 & -1\\
1 & 0
\end{pmatrix}
$ and so we obtain that $L_p$ is similar to the operator
\begin{eqnarray*}
\widetilde{L}_p:=P^{-1}L_pP =\begin{pmatrix}
\Delta & 0\\
0 & \Delta
\end{pmatrix}
-x\begin{pmatrix}
i & 0\\
0 & -i
\end{pmatrix},
\end{eqnarray*}
where
$P=
\begin{pmatrix}
1 & 1\\
-i & i
\end{pmatrix}.
$
Hence the semigroup generated by $L_p$ is analytic if and only if the semigroups generated by $\Delta \pm ix$ are analytic on $L^p(\CR)$.

To see that the semigroup generated by $B:=\Delta -ix$ is not analytic on $L^p(\CR)$ we introduce the transformation
\begin{eqnarray*}
\mathcal{U}_\sigma f(x)=f(x-\sigma),\qquad\;\, x\in \CR,\;\,{f\in L^p(\CR),}
\end{eqnarray*}
for arbitrary fixed $\sigma \in \CR $. So, we have
\begin{eqnarray*}
\mathcal{U}_{-\sigma}B\mathcal{U}_\sigma=B-i\sigma I.
\end{eqnarray*}
Hence,
\begin{eqnarray*}
\mathcal{U}_{-\sigma}(\mu -i\sigma-B)^{-1}\mathcal{U}_\sigma =(\mu-B)^{-1}
\end{eqnarray*}
and thus,
\begin{eqnarray*}
\|(\mu -i\sigma-B)^{-1}\|_{\cL(L^p(\CR))}=\|(\mu-B)^{-1}\|_{\cL(L^p(\CR))}
\end{eqnarray*}
 for arbitrary $\sigma \in \CR$ and every $\mu > 0$. Therefore, by \cite[Theorem II.4.6]{en00}
the semigroup generated by $B$ is not analytic.
\end{example}

Now, we collect some easy properties of the semigroup $\{S_p(t)\}$.

\begin{cor}\label{c.propertiesofS}
The following properties hold true.
\begin{enumerate}[\rm (a)]
\item The semigroup $\{S_p(t)\}$ is real, i.e.\ for $f \in L^p(\CR^d;\CR^m)$ we have
$S_p(t)f \in L^p(\CR^d; \CR^m)$.
\item The semigroups $\{S_p(t)\}$ are consistent, i.e.\ given $p,q \in (1,\infty)$ for $f\in L^p(\CR^d; \CC^m)\cap L^q(\CR^d; \CC^m)$ we have
$S_p(t)f=S_q(t)f$ for $t\geq 0$.
\end{enumerate}
\end{cor}

\begin{proof}
The semigroup $\{e^{tA_p}\}_{t\ge 0}$ is contractive, as is the multiplication semigroup $e^{tV_p}$ on $L^p(\CR^d;\CC^m)$.
We can thus use the Trotter product formula \cite[Corollary III.5.8]{en00} to conclude that
\begin{equation}\label{eq.tpf}
S_p(t)f = \lim_{n\to\infty} \Big(e^{n^{-1}tA_p}e^{n^{-1}tV_p}\Big)^n f
\end{equation}
for every $f \in L^p(\CR^d;\CC^m)$ and $t>0$. With the help of this formula, (a) and (b)  follow from the corresponding properties
of the semigroups generated by $A_p$ and $V_p$.
\end{proof}

We next address the case where $p=1$. We can easily extend the semigroups $\{S_p(t)\}$ to a consistent contraction semigroup
on $L^1(\CR^d; \CC^m)$. Note, however, that we no longer have knowledge of the domain of the generator.

\begin{thm}
\label{thm-L1}
There exists a strongly continuous semigroup of contractions $\{S_1(t)\}$ on $L^1(\CR^d;\CC^m)$ which is
consistent with every semigroup $\{S_p(t)\}$ for $p\in (1,\infty)$.
\end{thm}

\begin{proof}
Consider $f \in C_c^\infty(\CR^d;\CC^m)$. Then, $f \in L^p(\CR^d;\CC^m)$ for every $p \in (1,\infty)$ and by consistency
of the semigroups, we have $S_2(t)f = S_p(t)f \in L^p(\CR^d;\CC^m)$. Since $\{S_p(t)\}$ is a contraction, we obtain
\begin{eqnarray*}
\int_{\CR^d} |S_2(t)f|^p\, dx \leq \|f\|^p_p\leq \|f\|_1\|f\|_\infty^{p-1}.
\end{eqnarray*}
Using Fatou's lemma to let $p\to 1^+$, we infer that $S_2(t)f \in L^1(\CR^d;\CC^m)$ and $\|S_2(t)f\|_1 \leq
\|f\|_1$. Thus, by density, $S_2(t)$ can be uniquely
extended to an operator $S_1(t)$ on $L^1(\CR^d;\CC^m)$. By uniqueness of the extension, we obtain the semigroup
law for $\{S_1(t)\}$. By \cite[Proposition 4]{v92}, $\{S_1(t)\}$ is strongly continuous.

Noting that a function $f\in L^1(\CR^d;\CC^m)\cap L^2(\CR^d;\CC^m)$ can be approximated, simultaneously in $L^1$ and in
$L^2$, by a sequence of test functions, we see that $\{S_1(t)\}$ and $\{S_2(t)\}$ are consistent. Similarly, we see
that $\{S_1(t)\}$ and $\{S_p(t)\}$ are consistent for every $p\in (1,\infty)$.
\end{proof}

\begin{rem}\label{r.rescons}
From the proof of Theorem \ref{thm-L1} we do not obtain any information on the domain $D(L_1)$ of the generator
$L_1$ of the semigroup $\{S_1(t)\}$. However, as the semigroups $\{S_p(t)\}$ are consistent for $1\leq p < \infty$, so
are the resolvents $(\lambda - L_p)^{-1}$ of their generators for $\lambda >0$.
Indeed, if $f\in L^p(\CR^d;\CC^m)\cap L^q(\CR^d;\CC^m)$
for some $p,q \in [1,\infty)$, then for $g \in L^\infty(\CR^d;\CC^m)$ with compact support we have
\begin{align*}
\langle (\lambda - L_p)^{-1}f, g\rangle_{p,p'} & = \int_0^\infty e^{-\lambda t}\langle S_p(t)f, g\rangle_{p,p'} \, dt\\
&= \int_0^\infty e^{-\lambda t} \langle S_q(t)f, g\rangle_{q,q'}\, dt = \langle (\lambda-L_q)^{-1}f, g\rangle_{q,q'}.
\end{align*}
As $g$ was arbitrary, it follows that $(\lambda - L_p)^{-1}f = (\lambda - L_q)^{-1}f$.

From this it follows that if $f\in D(L_p)\cap L^q(\CR^d;\CC^m)$ ($1\le p,q<+\infty$) with $L_pf \in L^q(\CR^d;\CC^m)$,
then $f\in D(L_q)$ and $L_qf=L_pf$. In particular, we deduce that $C_c^\infty(\CR^d;\CC^m) \subset D(L_1)$.
\end{rem}

\section{Further properties of the semigroup}\label{s.further}

\subsection{Positivity}

We start by characterizing positivity of the semigroup.

\begin{prop}\label{p.positive}
The semigroup $\{S_p(t)\}$ is positive if and only if the off-diagonal entries of $V$ are nonnegative, i.e.\ $v_{kl}(x)\geq 0$ for almost every $x\in \CR^d$ whenever $k\neq l$.
\end{prop}

\begin{proof}
Let us recall from \cite[C-II Proposition 1.7]{aetal86} that the generator $G$ of a positive semigroup on a Banach lattice $X$ satisfies the positive minimum principle, i.e.\ for $0\leq x \in D(G)$ and $0\leq x^* \in X^*$ with $\langle x, x^*\rangle =0$ we have $\langle Gx, x^* \rangle \geq 0$. Let us denote  the canonical basis of $\CR^m$ by $(e_k)_{1\leq k\leq m}$.
Then, for $0\leq \varphi \in C_c^\infty(\CR^d)$ the nonnegative function $\varphi e_k$ belongs to $D(A_p)$ as well as to the dual space of $L^p(\CR^d;\CR^m)$.
Thus, if we assume that $\{S_p(t)\}$ is positive, it follows that for $k\neq l$ we have
\begin{align*}
0 & \leq \langle (A_p+V_p) \varphi e_l, \varphi e_k\rangle
= \int_{\CR^d}\div (Q \nabla \varphi)\varphi \langle e_l, e_k \rangle\, dx  + \int_{\CR^d} \varphi^2 \langle Ve_l, e_k\rangle\, dx\\
& = \int_{\CR^d}v_{kl}\varphi^2\, dx.
\end{align*}
As $\varphi$ is arbitrary, this implies that $v_{kl} \geq 0$ as claimed.

To prove the converse, assume that $v_{kl}(x) \geq 0$ for $k\neq l$ and almost every $x \in \CR^d$.
This is precisely the positive minimum principle for the matrix $V(x)$. For a bounded operator, the positive minimum principle is
not only necessary, but also sufficient to generate a positive semigroup, see \cite[C-II Theorem 1.11]{aetal86}.  Using this pointwise,
we see that the multiplication semigroup $e^{tV}$ is positive. As the semigroup $\{e^{tA_p}\}$ is positive, see \cite[Corollary 4.3]{o05}, the
positivity of the semigroup $\{S_p(t)\}$ follows once again from the Trotter product formula \eqref{eq.tpf}.
\end{proof}

\subsection{Ultracontractivity}
In this subsection we will establish  ultracontractivity of the semigroup $\{S_p(t)\}$. As a consequence, the semigroup
is given by an integral matrix kernel. In view of Corollary \ref{c.propertiesofS}(a), we confine ourselves to functions with values in $\CR^m$.
Since for $1\leq p <\infty$ the semigroups $\{S_p(t)\}$ are consistent, we drop the index $p$ and merely write $\{S(t)\}$ for our semigroup.
In what follows, we denote by $\{T_p(t)\}$ the scalar semigroup
on $L^p(\CR^d)$ generated by the scalar operator $D_p \coloneqq \div (Q\nabla \cdot )$,
defined on $W^{2,p}(\CR^d)$.
Note that also these semigroups are consistent, this is why also here we drop the index $p$.

 We start by the following technical lemma which gives a pointwise domination of $\{S(t)\}$.
\begin{lem}
Let Hypotheses \ref{hyp1} hold.
Then, for $f \in C_c^\infty (\CR^d; \CR^m)$, we have
\begin{equation}\label{pointwise sg domination}
|S(t)f|^2\le T(t)|f|^2,\qquad\;\, t>0 .
\end{equation}
\end{lem}

\begin{proof}
Let $f\in C_c^\infty(\CR^d;\CR^m)$ be given. Let us also fix $p\in (1,\infty)$.
We set $u(t,\cdot)=S(t)f$, for  $t\ge 0$. Since
$f\in D(A_q+V_q)$ is continuously embedded into $W^{2,q}(\CR^d;\CR^m)$,
$u$ belongs to $C([0,\infty);W^{2,q}(\CR^d;\CR^m))\cap C^1([0,\infty);L^q(\CR^d;\CR^m))$ for every $q\in [1,\infty)$. It thus follows that the scalar
function
$|u|^2$ belongs to $C([0,\infty);W^{2,p}(\CR^d))$. Since $u$ solves
the system of coupled partial differential equations $\partial_t u=(\div(Q\nabla u_k)-u_k)+Vu$, we get
\begin{align*}
\frac{1}{2} \partial_t |u|^2 = \langle\partial_t u,u\rangle &= \sum_{k=1}^{m}\div(Q\nabla u_k)u_k-|u|^2+\langle Vu,u\rangle \\
&\le \sum_{k=1}^{m}\sum_{i,j=1}^{d}\partial_i(q_{ij}\partial_j u_k)u_k -2|u|^2\\
&=\sum_{k=1}^{m}\sum_{i,j=1}^{d}\partial_i(q_{ij}u_k\partial_j u_k)-\sum_{k=1}^{m}\sum_{i,j=1}^{d}q_{ij}\partial_j u_k\partial_i u_k  -2|u|^2\\
&\le \frac{1}{2}\sum_{i,j=1}^{d}\partial_i(q_{ij}\partial_j |u|^2)-2|u|^2\\
&\le \frac{1}{2}D_p|u|^2.
\end{align*}
Thus, the function $v:=\partial_t |u|^2-D_p|u|^2$ belongs to $C([0,\infty);L^p(\CR^d))$ and is nonpositive. Fix $t>0$ and set $w(s,\cdot)=T(t-s)|u|^2(s,\cdot)$ for every $s\in [0,t]$.
As  is immediately seen,
\begin{align*}
\partial_sw(s,\cdot)=&-T(t-s)D_p|u|^2(s,\cdot)+T(t-s)\partial_s|u|^2(s,\cdot)\\
=&T(t-s)(\partial_s|u|^2(s,\cdot)-D_p|u|^2(s,\cdot))\\
=&T(t-s)v(s,\cdot)\le 0,
\end{align*}
since the semigroup $\{T(t)\}$ preserves positivity (see, again, \cite[Corollary 4.3]{o05}). Hence, $w(t,\cdot)\le w(0,\cdot)$, which is \eqref{pointwise sg domination}.
 \end{proof}

We can now establish ultracontractivity of the semigroup.

\begin{prop}
Assume that Hypotheses \ref{hyp1} hold.
Then there exists $M>0$ such that
\begin{equation}\label{ultracontractivity ineq.}
\|S(t)f\|_\infty \le M t^{-\frac{d}{2}}\|f\|_1,\qquad\;\,f\in L^1(\CR^d;\CR^m).
\end{equation}
Moreover, for every $t>0$ there exists a kernel $K(t,\cdot, \cdot)\in L^\infty(\CR^d\times\CR^d;\CR^{m\times m})$ such that
\begin{equation}\label{kernel representation}
(S_p(t)f)(x)=\int_{\CR^d}K(t,x,y)f(y)dy,\qquad\;\,x\in\CR^d,\;\,f\in L^p(\CR^d;\CR^m).
\end{equation}
\end{prop}

\begin{proof}
Let us first prove Estimate \eqref{ultracontractivity ineq.}. We fix $f\in C_c^\infty(\CR^d;\CR^m)$ and show that
\begin{equation}\label{L2-Loo estimate}
\|S(t)f\|_\infty\le M t^{-\frac{d}{4}}\|f\|_2,\qquad\;\,t>0.
\end{equation}
Throughout the proof $M$ is a constant, independent of $f$ and $t$, which may vary from line to line.
Using \eqref{pointwise sg domination} and  the ultracontractivity of the semigroup $\{T(t)\}$ we get
\begin{align*}
\|S(t)f\|_\infty^2\le \|T(t)|f|^2\|_\infty\le M t^{-\frac{d}{2}}\||f|^2\|_1=M t^{-\frac{d}{2}}\|f\|_2^2
\end{align*}
for $t>0$. Taking square roots, this shows \eqref{L2-Loo estimate}.
Next, we prove the $L^1$--$L^2$ estimate
\begin{equation}\label{L1-L2 estimate}
\|S(t)f\|_2\le Mt^{-\frac{d}{4}}\|f\|_1,\qquad\; t>0.
\end{equation}
To that end, note that the adjoint $V^*$ also satisfies Hypotheses \ref{hyp1}, except for the fact that instead of the
boundedness of $D_j V^* (-V^*)^{-\alpha}$, we obtain the boundedness of $(-V^*)^{-\alpha}D_j V^*$. However, an inspection
of the proofs above shows that they remain valid also under this assumption, whence we obtain the same results for the
adjoint semigroup $\{S^*(t)\}$.
In particular \eqref{pointwise sg domination} and thus also \eqref{L2-Loo estimate} hold true for  $\{S^*(t)\}$. Consequently,
\begin{align*}
\|S(t)f\|_2 &=\sup_{\|\varphi\|_{2}=1}\langle S(t)f,\varphi\rangle_{2,2}
=\sup_{\|\varphi\|_{2}=1}\langle f,S^*(t)\varphi\rangle_{2,2}\\
& \le \sup_{\|\varphi\|_{2}=1}\|S^*(t)\varphi\|_\infty \|f\|_{1}
\le Mt^{-\frac{d}{4}}\|f\|_{1}
\end{align*}
and \eqref{L1-L2 estimate} follows.
By the semigroup law and Estimates \eqref{L2-Loo estimate} and \eqref{L1-L2 estimate} we obtain
\begin{align*}
\|S(t)f\|_\infty = \|S(t/2)S(t/2)f\|_{\infty}\le Mt^{-\frac{d}{4}}\|S(t/2)f\|_2 \leq Mt^{-\frac{d}{2}}\|f\|_1
\end{align*}
for $t>0$. Using the density of $C_c^\infty(\CR^d; \CR^m)$ in $L^1(\CR^d;\CR^m)$,
we can easily complete the proof of \eqref{ultracontractivity ineq.}.\medskip

We next establish the existence of the kernel. We fix $t>0$, $f=(f_1,\dots,f_m)$ and denote the canonical basis
of $\CR^m$ by $\{e_i\}_{1\leq i\leq m}$. Then, we have
\begin{eqnarray*} S(t)f=\sum_{j=1}^{m}S(t)(f_j e_j)=\sum_{i,j=1}^{m}\langle S(t)(f_j e_j),e_i\rangle e_i. \end{eqnarray*}
For $i,j\in\{1,\dots,m\}$ and $u\in L^1(\CR^d)$, let $S_{i,j}(t)u=\langle S(t)(u e_j),e_i\rangle$.
Using \eqref{ultracontractivity ineq.}, we obtain
\begin{align*}
\|S_{i,j}(t)u\|_\infty=  \|\langle S(t)(u e_j),e_i\rangle\|_\infty\leq\|S(t)(u e_j)\|_\infty\leq
 Mt^{-\frac{d}{2}}\|ue_j\|_1=Mt^{-\frac{d}{2}}\|u\|_1.
\end{align*}
Thus, $S_{i,j}(t)$ maps $L^1(\CR^d)$ into $L^\infty(\CR^d)$. By the Dunford--Pettis theorem, see \cite[Theorem 1.3]{ab94}, there
exists a kernel $k_{i,j}(t,\cdot, \cdot )\in L^\infty(\CR^d\times\CR^d)$ such that
\begin{eqnarray*} (S_{i,j}(t)u)(x)=\int_{\CR^d}k_{ij}(t,x,y)u(y)dy, \end{eqnarray*}
for all $x\in \CR^d$. Setting $K(t, \cdot, \cdot) = (k_{ij}(t,\cdot, \cdot))_{i,j=1}^m$, we conclude that
 \begin{eqnarray*}
 S(t)f=\int_{\CR^d}\sum_{i,j=1}^{m}k_{ij}(t,x,y)f_j(y)e_i dy=\int_{\CR^d}K(t,x,y) f(y)dy,
 \end{eqnarray*}
 for all $f\in L^p(\CR^d;\CR^m)$.
\end{proof}

\subsection{Spectrum of the generator}

Last, we analyze the spectrum of the operator $L_p$ and prove the following result.

\begin{thm}\label{t.compactness}
Assume in addition to Hypotheses \ref{hyp1} that there exists a function $\kappa: \CR^d \to [0,\infty) $ with $\lim_{|x|\to \infty}
\kappa (x) = \infty$ such that $|V(x)\xi| \geq \kappa (x)|\xi|$ for all $x\in\CR^d$ and $\xi \in \CR^m$. Then, for all $p\in (1,\infty)$, the operator $L_p$ has compact resolvent.
Consequently, its spectrum is independent of $p\in (1,\infty)$ and consists of eigenvalues only.
\end{thm}

\begin{proof}
Fix a $p\in (1,\infty)$. To show that $L_p$ has compact resolvent, it suffices to prove that $D(L_p)$ is compactly embedded into $L^p(\CR^d;\CC^m)$.
Note that, as a consequence of Theorem \ref{t.main1}, the graph norm of $L_p$ is equivalent to the norm $u\mapsto \tnorm{u} \coloneqq  \|u\|_{2.p}+\|V_pu\|_p$. Indeed, for $u \in D(L_p)$ we clearly have $\|\cdot\|_{D(A)}\le C\tnorm{u}$ for some positive constant $C$, independent of $u$. As $D(L_p)$ is complete with respect to both norms,
the equivalence of the two norms follows from the open mapping theorem. In view of this remark, in the rest of
the proof we assume that $D(A_p+V_p)$ is endowed with this norm.

By our additional assumption, we have
\begin{equation}
\|V_pu\|_p^p \geq \int_{\CR^d}\kappa (x)^p|u(x)|^p\, dx
\label{freccia}
\end{equation}
for every $u\in D(A_p+V_p)$.
Using this estimate it is easy to check that the closed unit ball of $D(A_p+V_p)$ is compact (or, equivalently, totally bounded) in $L^p(\CR^d;\CR^m)$. To see this, let $u$ belong to the unit ball of $D(A_p + V_p)$ so that in particular $\|V_pu\|_p\leq 1$.
Given $\varepsilon>0$ we fix $R>0$ sufficiently large so that $\kappa\ge \varepsilon^{-1}$ outside the ball $B_R \coloneqq\{ x\in \CR^d : |x| < R\}$. Then, from Equation \eqref{freccia}, we deduce that
\begin{align*}
\int_{\CR^d\setminus B_R}|u(x)|^p\, dx\le &\varepsilon^p\int_{\CR^d\setminus B_R}\kappa(x)^p|u(x)|^p\,dx\\
\le &\varepsilon^p\int_{\CR^d}\kappa(x)^p|u(x)|^p\,dx
\le \varepsilon^p\|V_pu\|_p^p\le \varepsilon^p.
\end{align*}
Since the set of the restriction to $B_R$ of functions in $D(A_p+V_p)$ is continuously embedded in
$W^{2,p}(B_R;\CC^m)$, which is compactly embedded into $L^p(B_R; \CC^m)$, we find finitely many functions $g_1,\ldots,g_k\in L^p(B_R;\CC^m)$ such that, for
every $u$ in the unit ball of $D(A_p+V_p)$, there
exists an index $j\in \{1,\ldots, k\}$ such that
\begin{eqnarray*}
\int_{B_R}|u(x)-g_j(x)|^p\,dx\le\varepsilon^p.
\end{eqnarray*}
Denoting the trivial extension of $g_j$ to $\CR^d$ by $\bar g_j$, we have
\begin{align*}
\int_{\CR^d}|u(x)-\bar g_j(x)|^p\,dx=\int_{B_R}|f(x)-g_j(x)|^p\,dx
+\int_{\CR^d\setminus B_R}|u(x)|^p\,dx\le 2\varepsilon^p.
\end{align*}
This shows that the unit ball of $D(A_p+V_p)$ is covered by  the balls in $L^p(\CR^d;\CC^m)$ centered at $\overline g_j$ of radius $2^{\frac{1}{p}}\varepsilon$. As $\eps >0$ was arbitrary, it follows that the unit ball of $D(A_p+V_p)$ is totally bounded
in $L^p(\CR^d;\CC^m)$.

The fact that the spectrum consists only of eigenvalues follows from spectral properties of compact operators with help of the spectral mapping theorem for the resolvent, cf. \cite[Theorem IV.1.13]{en00}.

Since the resolvent operators $(\lambda-L_p)^{-1}$ are consistent (see Remark \ref{r.rescons}) and compact, the $p$-independence of the spectrum
follows from \cite[Corollary 1.6.2]{d89}.
\end{proof}

Using the Cauchy--Schwarz inequality, we see that  the additional assumption in Theorem \ref{t.compactness} is, in particular,
satisfied if we have $\la V(x)\xi, \xi\ra\le -\tilde\kappa (x)|\xi|^2$ for a certain function $\tilde \kappa    : \CR^d \to [0,\infty)$ (with $\lim_{|x|\to\infty}\tilde\kappa (x) = \infty$)
for all $\xi \in\CR^m$. If $V(x)$ is symmetric for every $x\in\CR^d$, then the two conditions are equivalent. Indeed, the assumption in Theorem \ref{t.compactness} and Hypotheses \ref{hyp1}(b) imply that every eigenvalue of $V(x)$ {does not exceed $-\kappa(x)$. This, in turn, is equivalent to the condition
$\la V(x)\xi, \xi\ra \le-\kappa (x)|\xi|^2$, for all $\xi\in\CR^m$.

However, the assumption in Theorem \ref{t.compactness} is more general than this, since it is for example satisfied for the potential in Example \ref{ex.poly}.
Indeed, in this case we have $|V(x)\xi| = {(1+|x|^r)}|\xi|$, so that we can choose $\kappa (x)={1+|x|^r}$ for all $x\in\CR^d$. On the other hand,
we have $\la V(x)\xi, \xi\ra = 0$ for all $x\in \CR$ and $\xi \in \CR^2$.

Finally, we note that compactness may fail even if all entries in the potential $V$ are unbounded near $\infty$.

\begin{example}
Consider the potential
\begin{eqnarray*}
V(x) =
\begin{pmatrix}
-|x| & |x|\\
|x| & -|x|
\end{pmatrix},\qquad\;\,x\in\CR^d.
\end{eqnarray*}
Note that $\la V(x)\xi, \xi\ra = - |x|(\xi_1+\xi_2)^2 \leq 0$ for all $x\in \CR^d$ and $\xi\in\CR^2$, so that the quadratic form is semibounded.
Moreover, as the entries of $V$ are Lipschitz continuous, $\nabla V$ is uniformly bounded. Thus, choosing $Q= I$,  Hypotheses \ref{hyp1} are satisfied with $\alpha =0$.
 However, $0$ is an eigenvalue of $V(x)$ with eigenvector $(1,1)^\mathsf{T}$  for every $x\in \CR^d$. Thus, if we pick $f\in W^{2,p}(\CR^d;\CC^m)$, we see that $u=(f,f) \in D(A_p+V_p)$ and $(A_p+V_p)u = (\Delta_p f, \Delta_p f)$,
 where $\Delta_p$ is the scalar Laplace operator on $L^p(\CR^d)$.
It follows that $(\lambda-A_p-V_p)^{-1}u=((\lambda-\Delta_p)^{-1}f,(\lambda-\Delta_p)^{-1}f)$ for $\lambda\in\rho(A_p+V_p)\cap\rho(\Delta_p)$, $u=(f,f)\in L^p(\CR^d;\CC^2)$. As a consequence, we obtain
 \begin{eqnarray*}
 S_p(t)u
 =
 \begin{pmatrix}
 e^{t\Delta_p} f\\
 e^{t\Delta_p}f
 \end{pmatrix},\qquad\;\,t>0,\;\,f\in L^p(\CR^d),
 \end{eqnarray*}
 where $\{e^{t\Delta_p}\}$ is the (scalar valued) Gaussian semigroup.
 As the Gaussian semigroup is not compact, the semigroup $\{S_p(t)\}$ cannot be compact either.
\end{example}


\begin{thebibliography}{10}

\bibitem{aalt}
{\sc D.~Addona, L.~Angiuli, L.~Lorenzi, and M.~Tessitore}, {\em On coupled
  systems of kolmogorov equations with applications to stochastic differential
  games}, ESAIM Control. Optim. Calc. Var. (to appear),
  doi.org/10.1051/cocv/2016019,  (2017).

\bibitem{alp16}
{\sc L.~Angiuli, L.~Lorenzi, and D.~Pallara}, {\em {$L^p$}-estimates for
  parabolic systems with unbounded coefficients coupled at zero and first
  order}, J. Math. Anal. Appl., 444 (2016), pp.~110--135.

\bibitem{ab94}
{\sc W.~Arendt and A.~V. Bukhvalov}, {\em Integral representations of
  resolvents and semigroups}, Forum Math., 6 (1994), pp.~111--135.

\bibitem{aetal86}
{\sc W.~Arendt, A.~Grabosch, G.~Greiner, U.~Groh, H.~P. Lotz, U.~Moustakas,
  R.~Nagel, F.~Neubrander, and U.~Schlotterbeck}, {\em One-parameter semigroups
  of positive operators}, vol.~1184 of Lecture Notes in Mathematics,
  Springer-Verlag, Berlin, 1986.

\bibitem{aetal02}
{\sc P.~Auscher, S.~Hofmann, M.~Lacey, A.~McIntosh, and P.~Tchamitchian}, {\em
  The solution of the {K}ato square root problem for second order elliptic
  operators on {${\mathbb{R}}^n$}}, Ann. of Math. (2), 156 (2002),
  pp.~633--654.

\bibitem{BGT}
{\sc V.~Betz, B.~D. Goddard, and S.~Teufel}, {\em Superadiabatic transitions in
  quantum molecular dynamics}, Proc. R. Soc. A, 465 (2009), pp.~3553--3580.

\bibitem{b86}
{\sc D.~L. Burkholder}, {\em Martingales and {F}ourier analysis in {B}anach
  spaces}, in Probability and analysis ({V}arenna, 1985), vol.~1206 of Lecture
  Notes in Math., Springer, Berlin, 1986, pp.~61--108.

\bibitem{Dall}
{\sc G.~M. Dall'Ara}, {\em Discreteness of the spectrum of {S}chr\"odinger
  operators with non-negative matrix-valued potentials}, J. Funct. Anal., 268
  (2015), pp.~3649--3679.

\bibitem{d89}
{\sc E.~B. Davies}, {\em Heat kernels and spectral theory}, vol.~92 of
  Cambridge Tracts in Mathematics, Cambridge University Press, Cambridge, 1989.

\bibitem{dl11}
{\sc S.~Delmonte and L.~Lorenzi}, {\em On a class of weakly coupled systems of
  elliptic operators with unbounded coefficients}, Milan J. Math., 79 (2011),
  pp.~689--727.

\bibitem{ds97}
{\sc X.~T. Duong and G.~Simonett}, {\em {$H_\infty$}-calculus for elliptic
  operators with nonsmooth coefficients}, Differential Integral Equations, 10
  (1997), pp.~201--217.

\bibitem{en00}
{\sc K.-J. Engel and R.~Nagel}, {\em One-parameter semigroups for linear
  evolution equations}, vol.~194 of Graduate Texts in Mathematics,
  Springer-Verlag, New York, 2000.
\newblock With contributions by S. Brendle, M. Campiti, T. Hahn, G. Metafune,
  G. Nickel, D. Pallara, C. Perazzoli, A. Rhandi, S. Romanelli and R.
  Schnaubelt.

\bibitem{h06}
{\sc M.~Haase}, {\em The functional calculus for sectorial operators}, vol.~169
  of Operator Theory: Advances and Applications, Birkh\"auser Verlag, Basel,
  2006.

\bibitem{Ha-Rh}
{\sc T.~Hansel and A.~Rhandi}, {\em The {O}seen-{N}avier-{S}tokes flow in the
  exterior of a rotating obstacle: {T}he non-autonomous case}, J. Reine Angew.
  Math., 694 (2014), pp.~1--26.

\bibitem{Ha-He}
{\sc F.~Haslinger and B.~Helffer}, {\em Compactness of the solution operator to
  $\overline{\partial}$ in weighted ${L}^2$-spaces}, J. Funct. Anal., 243
  (2007), pp.~679--697.

\bibitem{hetal09}
{\sc M.~Hieber, L.~Lorenzi, J.~Pr{\"u}ss, A.~Rhandi, and R.~Schnaubelt}, {\em
  Global properties of generalized {O}rnstein-{U}hlenbeck operators on
  {$L^p(\mathbb{R}^N,\mathbb{R}^N)$} with more than linearly growing
  coefficients}, J. Math. Anal. Appl., 350 (2009), pp.~100--121.

\bibitem{HRS}
{\sc M.~Hieber, A.~Rhandi, and O.~Sawada}, {\em The {N}avier-{S}tokes flow for
  globally {L}ipschitz continuous initial data}, Res. Inst. Math. Sci. (RIMS),
  (2007), pp.~159--165.
\newblock Kyoto Conference on the Navier-Stokes Equations and their
  Applications, RIMS Kkyroku Bessatsu, B1.

\bibitem{hieber}
{\sc M.~Hieber and O.~Sawada}, {\em The {N}avier-{S}tokes equations in
  {$\mathbb{R}^n$} with linearly growing initial data}, Arch. Ration. Mech.
  Anal., 175 (2005), pp.~269--285.

\bibitem{lb07}
{\sc L.~Lorenzi}, {\em Analytical methods for {K}olmogorov equations. Second
  edition}, Monograph and research notes in Mathematics, Chapman \& Hall/CRC,
  Boca Raton, FL, 2017.

\bibitem{mp97}
{\sc S.~Monniaux and J.~Pr{\"u}ss}, {\em A theorem of the {D}ore-{V}enni type
  for noncommuting operators}, Trans. Amer. Math. Soc., 349 (1997),
  pp.~4787--4814.

\bibitem{o05}
{\sc E.~M. Ouhabaz}, {\em Analysis of heat equations on domains}, vol.~31 of
  London Mathematical Society Monographs Series, Princeton University Press,
  Princeton, NJ, 2005.

\bibitem{prs06}
{\sc J.~Pr\"uss, A.~Rhandi, and R.~Schnaubelt}, {\em The domain of elliptic
  operators on {$L^p(\mathbb{R}^d)$} with unbounded drift coefficients},
  Houston J. Math., 32 (2006), pp.~563--576.

\bibitem{ps90}
{\sc J.~Pr{\"u}ss and H.~Sohr}, {\em On operators with bounded imaginary powers
  in {B}anach spaces}, Math. Z., 203 (1990), pp.~429--452.

\bibitem{v92}
{\sc J.~Voigt}, {\em One-parameter semigroups acting simultaneously on
  different {$L_p$}-spaces}, Bull. Soc. Roy. Sci. Li\`ege, 61 (1992),
  pp.~465--470.

\end{thebibliography}
\end{document}